\documentclass[11pt]{article}

\usepackage{amsmath,amssymb,amsbsy,amsfonts,amsthm,latexsym,
            amsopn,amstext,amsxtra,euscript,amscd,color,breqn}

\PassOptionsToPackage{hyphens}{url}\usepackage{hyperref}
    
\usepackage[capbesideposition=outside,capbesidesep=quad]{floatrow}

\restylefloat{table}
            
\usepackage{multirow,caption}

\newfont{\teneufm}{eufm10}
\newfont{\seveneufm}{eufm7}
\newfont{\fiveeufm}{eufm5}

\usepackage{systeme}

\usepackage{xcolor}
\usepackage{tikz}
\usetikzlibrary{shapes}
\usetikzlibrary{decorations.markings}

\usepackage{pstricks}
\usetikzlibrary{shapes,arrows}
 \usepackage{pstricks-add}
\usepackage{epsfig}
 \usepackage[space]{grffile} 
 \usepackage{etoolbox} 
 \makeatletter 
 \patchcmd\Gread@eps{\@inputcheck#1 }{\@inputcheck"#1"\relax}{}{}
 \makeatother

\newtheorem{thm}{Theorem}
\newtheorem{lem}[thm]{Lemma}

\newtheorem{conjecture}[thm]{Conjecture}

\newcommand{\Tr}{{\rm Tr}}

\newcommand{\cB}{\mathscr{B}}

\def\+{\oplus}

\def\cB{{\mathcal B}}

\def\F{{\mathbb F}}

\def\Fn{{\mathbb{F}_{p^n}}}

\def\00{{\bf 0}}
\def\11{{\bf 1}}
\def\+{\oplus}

\def\\{\cr}
\def\({\left(}
\def\){\right)}

\newcommand{\cardinality}[1]{\# #1}

\usepackage[colorinlistoftodos,prependcaption,textsize=tiny]{todonotes}

\providecommand{\newoperator}[3]{%
  \newcommand*{#1}{\mathop{#2}#3}}

\newoperator{\FD}{\mathrm{FD}}{\nolimits}

\begin{document}
\title{\bf Low $c$-differential and $c$-boomerang uniformity of the swapped inverse function}
\author{Pantelimon~St\u anic\u a  \\ 
Applied Mathematics Department, \\
Naval Postgraduate School, Monterey, USA. \\
E-mail: pstanica@nps.edu}

\maketitle

\begin{abstract}
Modifying the binary inverse function in a variety of ways, like swapping two output points has been known to produce a $4$-differential uniform permutation function. Recently, in~\cite{Li19} it was shown that this swapped version of the inverse function has boomerang uniformity exactly $10$, if $n\equiv 0\pmod 6$, $8$,  if $n\equiv 3\pmod 6$, and 6,  if $n\not\equiv 0\pmod 3$.  Based upon the $c$-differential notion we defined in~\cite{EFRST20} and $c$-boomerang uniformity from~\cite{S20}, in this paper we characterize the $c$-differential and $c$-boomerang uniformity for the $(0,1)$-swapped  inverse function in characteristic~$2$: we show that for all~$c\neq 1$, the $c$-differential uniformity is upper bounded by~$4$ and the $c$-boomerang uniformity by~$5$ with both bounds being attained for~$n\geq 4$.
\end{abstract}


\section{Introduction and basic definitions}

 The boomerang attack on block ciphers developed by Wagner~\cite{Wag99} was quickly picked by several works~\cite{BK99,KKS00,BDK02,Kim12} in the applied realm, but it took almost ten years until  Cid et al.~\cite{Cid18} came up with a theoretical tool called  the Boomerang Connectivity Table (BCT), and further Boura and Canteaut~\cite{BC18} defined the boomerang uniformity to theoretical investigate the functions' resistance to the boomerang attack.
 Using  a multiplier differential we introduced in~\cite{EFRST20},  we extended the notion of  the $c$-Boomerang Connectivity Table ($c$-BCT) and $c$-boomerang uniformity  in~\cite{S20}.
 
  Here, we continue the work by investigating the well-known swapped inverse, proposed in~\cite{Yu13}, which was the subject of many papers since then (see~\cite{CV19,QT13,QT16,TCT14,Za14}, to cite just a few papers; a generalization allowing the modification of any two output values is done in \cite{Kal2019}).

As customary,  we let:
\begin{enumerate}
\item  $m,n$ are positive integers  and $p$ a prime number, $\cardinality{S}$  denotes the cardinality of a set $S$ and $\bar z$ is the complex conjugate;
\item  $\F_{p^n}$ is the  finite field with $p^n$ elements, and $\F_{p^n}^*=\F_{p^n}\setminus\{0\}$ is the multiplicative group (for $a\neq 0$, we often write $\frac{1}{a}$ to mean the inverse of $a$ in the multiplicative group);
\item  We let $\F_p^n$ be the $n$-dimensional vector space over $\F_p$.  Next,
$f:\F_{p^n}$ (or $\F_p^n$) to $\F_p$  is a {\em $p$-ary  function} on $n$ variables. Any map $F:\F_{p^n}\to\F_{p^m}$ (or, $\F_p^n\to\F_p^m$)  is called a {\em vectorial $p$-ary  function}, or {\em $(n,m)$-function}. 
\item When $m=n$, $F$ can be uniquely represented as a univariate polynomial over $\F_{p^n}$ (using some identification, via a basis, of the finite field with the vector space) of the form
$
F(x)=\sum_{i=0}^{p^n-1} a_i x^i,\ a_i\in\F_{p^n},
$
whose {\em algebraic degree}   is then the largest Hamming weight of the exponents $i$ with $a_i\neq 0$. 
\item  
Given a $p$-ary  function $f$, the derivative of $f$ with respect to~$a \in \F_{p^n}$ is the $p$-ary  function
$
 D_{a}f(x) =  f(x + a)- f(x), \mbox{ for  all }  x \in \F_{p^n},
$
which can be naturally extended to vectorial $p$-ary functions.
\item
While the next concept can be defined for general $(n,m)$-functions, in this paper we consider $m=n$.
For an $(n,n)$-function $F$, and $a,b\in\F_{p^n}$, we let $\Delta_F(a,b)=\cardinality{\{x\in\F_{p^n} : F(x+a)-F(x)=b\}}$. We call the quantity
$\delta_F=\max\{\Delta_F(a,b)\,:\, a,b\in \F_{p^n}, a\neq 0 \}$ the {\em differential uniformity} of $F$. If $\delta_F= \delta$, then we say that $F$ is differentially $\delta$-uniform. 
\item If $\delta=1$, then $F$ is called a {\em perfect nonlinear} ({\em PN}) function, or {\em planar} function. If $\delta=2$, then $F$ is called an {\em almost perfect nonlinear} ({\em APN}) function. It is well known that PN functions do not exist if $p=2$.
\end{enumerate}
The reader can consult~\cite{Bud14,CH1,CH2,CS17,MesnagerBook,Tok15} for more on  Boolean and $p$-ary functions.

The initial concept of boomerang uniformity was defined for permutations (of course, proper $S$-boxes)  in the following way.
Let $F$ be a permutation on $\F_{2^n}$ and $(a,b)\in\F_{2^n}\times \F_{2^n}$. One defines the entries of the {\em Boomerang Connectivity Table} \textup{(}{\em BCT}\textup{)} by
\[
\cB_F(a,b)=\cardinality \{x\in\F_{2^n}|F^{-1} (F(x)+b)+F^{-1}(F(x+a)+b)=a  \},
\]
where $F^{-1}$ is the compositional inverse of $F$. The {\em boomerang uniformity} of $F$ is defined as
\[
\beta_F=\max_{a,b\in\Fn*} \cB_F(a,b).
\]
We also say that $F$ is a $\beta_F$-uniform BCT function. Surely, $\Delta_F(a,b)=0,2^n$ and $\cB_F(a,b)=2^n$  whenever $ab=0$. We know that $\delta_F=\delta_{F^{-1}}$, $\beta_F=\beta_{F^{-1}}$, and  for permutations, $\beta_F\geq \delta_F$ and they are equal for APN permutations.  Further, from~\cite{Li19} we know that  for quadratic permutations,
$
\delta_F\leq \beta_F\leq \delta_F(\delta_F-1).
$
 
 We mention here that this concept was quickly picked up and studied in~\cite{BC18, BPT19, CV19,Li19, LiHu20, Mes19,TX20}, to mention just a few recent papers.  
 
  Li et al.~\cite{Li19} (see also~\cite{Mes19}) observed that
\begin{equation*}
\begin{split}
\label{eq:boom-diff}
\cB_F(a,b)&=\cardinality \left\{ (x,y)\in\Fn\times \Fn\,\Large\big|\, \substack{F(x)+F(y)=b \\  F(x+a)+F(y+a)=b } \right\}, \\
 &\stackrel{y=x+\gamma}{=}\cardinality \left\{ (x,\gamma)\in\Fn\times \Fn\,\Large\big|\, \substack{F(x+\gamma)+F(x)=b \\ F(x+\gamma+a)+F(x+a)=b }   \right\}\\
& =\sum_{\gamma\in\Fn} \cardinality  \left\{ x \in\Fn \,\big|\, D_\gamma F(x)=b \text{ and }  D_\gamma F(x+a)=b   \right\},
\end{split}
\end{equation*}
which allowed the concept to be extended to non-permutations, since it avoids the inverse of~$F$.

We extended this notion recently in~\cite{S20} to $c$-boomerang connectivity table, based upon our prior~\cite{EFRST20} $c$-differential concept (see~\cite{RS20,YZ20} for more work related to it).
 
For a $p$-ary $(n,m)$-function   $F:\F_{p^n}\to \F_{p^m}$, and $c\in\F_{p^m}$, the ({\em multiplicative}) {\em $c$-derivative} of $F$ with respect to~$a \in \F_{p^n}$ is the  function
\[
 _cD_{a}F(x) =  F(x + a)- cF(x), \mbox{ for  all }  x \in \F_{p^n}.
\]

For an $(n,n)$-function $F$, and $a,b\in\F_{p^n}$, we let the entries of the $c$-Difference Distribution Table ($c$-DDT) be defined by ${_c\Delta}_F(a,b)=\cardinality{\{x\in\F_{p^n} : F(x+a)-cF(x)=b\}}$. We call the quantity
\[
\delta_{F,c}=\max\left\{{_c\Delta}_F(a,b)\,|\, a,b\in \F_{p^n}, \text{ and } a\neq 0 \text{ if $c=1$} \right\}\]
the {\em $c$-differential uniformity} of~$F$. If $\delta_{F,c}=\delta$, then we say that $F$ is differentially $(c,\delta)$-uniform (or that $F$ has $c$-uniformity $\delta$, or for short, {\em $F$ is $\delta$-uniform $c$-DDT}). If $\delta=1$, then $F$ is called a {\em perfect $c$-nonlinear} ({\em PcN}) function (certainly, for $c=1$, they only exist for odd characteristic $p$; however, as proven in~\cite{EFRST20}, there exist PcN functions for $p=2$, for all  $c\neq1$). If $\delta=2$, then $F$ is called an {\em almost perfect $c$-nonlinear} ({\em APcN}) function. 
When we need to specify the constant $c$ for which the function is PcN or APcN, then we may use the notation $c$-PN, or $c$-APN.
It is easy to see that if $F$ is an $(n,n)$-function, that is, $F:\F_{p^n}\to\F_{p^n}$, then $F$ is PcN if and only if $_cD_a F$ is a permutation polynomial.

Further, 
for an $(n,n)$-function $F$, $c\neq 0$, and $(a,b)\in\Fn\times \Fn$,  we define the {\em $c$-Boomerang Connectivity Table}  \textup{(}$c$-BCT\textup{)} entry at $(a,b)$ to be
{\small
\begin{equation*}
\label{eq:originalBCT}
  _c\cB_F(a,b)=\cardinality \left\{ x\in\Fn\,\Big|\,  F^{-1}(c^{-1} F(x+a)+b) -F^{-1}(cF(x)+b)=a \right\}.
\end{equation*}
}
Further, the {\em $c$-boomerang uniformity} of $F$ is defined by
\[
\beta_{F,c}=\max_{a,b\in\Fn*} {_c}\cB_F(a,b).
\]
If $\beta_{F,c}=\beta$, we also say that $F$ is a $\beta$-uniform $c$-BCT function.
We showed in~\cite{S20} that we can avoid inverses, thus   allowing the definition to be extended to all $(n,m)$-function, not only permutations. Precisely,   the entries of the $c$-Boomerang Connectivity Table at $(a,b)\in\F_{p^n}\times\F_{p^n}$  can be given by
\begin{align*}
_c\cB_F(a,b)&=\cardinality \left\{ (x,\gamma)\in\Fn\times \Fn\,\Big|\, \Large\substack{F(x+\gamma)-cF(x) =b \\  F(x+\gamma+a)- c^{-1} F(x+a)=b }  \right\}\\
& =\sum_{\gamma\in\Fn} \cardinality  \left\{ x \in\Fn \,\big|\, \Large \substack{_cD_\gamma F(x)=b \text{ and }  _{c^{-1}}D_\gamma F(x+a)=b \\
\text{the $c$-boomerang system}}  \right\}.
\end{align*}

  Besides various connections between $c$-DDT and $c$-BCT and characterizations via Walsh transforms, we investigated some of the known perfect nonlinear and the inverse function in all characteristics, in~\cite{S20}. For example, we showed that in general, if $F(x)=x^{2^n-2}$ on $\F_{2^n}$, ${_c}\cB_F(a,b)\leq 3$, and if $F(x)=x^{p^n-2}$ on $\F_{p^n}$ ($p$ odd), ${_c}\cB_F(a,b)\leq 4$ and gave complete conditions when the upper bound happens.
    It is the purpose of this paper to investigate both the $c$-DDT and $c$-BCT for the $(0,1)$-output swapped inverse function in the binary case (which is  a $4$-uniform DDT permutation function occurring in many papers~\cite{Li13,Li19,PT17,QT13,QT16}, to cite just a few works).
  
  The rest of the paper is organized as follows.  Section~\ref{sec4} and~\ref{sec5} deal  with $c$-differential uniformity, respectively, $c$-boomerang uniformity of the swapped function. Section~\ref{sec6}   concludes the paper.

\section{The $c$-differential uniformity of the swapped inverse function}
 \label{sec4}
 
  We will be using  throughout Hilbert's Theorem 90 (see~\cite{Bo90}), which states that if $\mathbb{F}\hookrightarrow \mathbb{K}$  is a cyclic Galois extension and $\sigma$ is a generator of the Galois group ${\rm Gal}(\mathbb{K}/\mathbb{F})$, then for $x\in \mathbb{K}$, the relative trace $\Tr_{\mathbb{K}/\mathbb{F}}(x)=0$ if and only if $x=\sigma(y)-y$, for some $y\in\mathbb{K}$.
We also need the following two lemmas.  
\begin{lem} 
\label{lem10} 
Let $n$ be a positive integer. We have:
\begin{enumerate}
 \item[$(i)$] The equation
$x^2 + ax + b = 0$, with $a,b\in\F_{2^n}$, $a\neq 0$,
has two solutions in $\F_{2^n}$ if  $\Tr\left(
\frac{b}{a^2}\right)=0$, and zero solutions otherwise \textup{(}see~\textup{\cite{BRS67}}\textup{)}.
\item[$(ii)$]  The equation
$x^2 + ax + b = 0$, with $a,b\in\F_{p^n}$, $p$ odd,
has (two, respectively, one) solutions in $\F_{p^n}$ if and only if the discriminant $a^2-4b$ is a (nonzero, respectively, zero) square in $\F_{p^n}$.
\end{enumerate}
\end{lem}

\begin{lem}[\textup{\cite{EFRST20}}]
\label{lem:gcd}
Let $p,k,n$ be integers greater than or equal to $1$ (we take $k\leq n$, though the result can be shown in general). Then
\begin{align*}
&  \gcd(2^{k}+1,2^n-1)=\frac{2^{\gcd(2k,n)}-1}{2^{\gcd(k,n)}-1},  \text{ and if  $p>2$, then}, \\
& \gcd(p^{k}+1,p^n-1)=2,   \text{ if $\frac{n}{\gcd(n,k)}$  is odd},\\
& \gcd(p^{k}+1,p^n-1)=p^{\gcd(k,n)}+1,\text{ if $\frac{n}{\gcd(n,k)}$ is even}.\end{align*}
Consequently, if either $n$ is odd, or $n\equiv 2\pmod 4$ and $k$ is even,   then $\gcd(2^k+1,2^n-1)=1$ and $\gcd(p^k+1,p^n-1)=2$, if $p>2$.
\end{lem}

Given $F:\F_{p^n}\to\F_{p^n}$, and two points $x_0\neq x_1$ in $\F_{p^n}$, we let  $G_{x_0x_1}$ be the $\{x_0,x_1\}$-swapping of $F$ defined by
\begin{equation*}
G_{x_0x_1}(x)=F(x)+\left( (x+x_0)^{p^n-1}+(x+x_1)^{p^n-1}\right)(y_0+y_1),
\end{equation*}
where $y_0=F(x_0),y_1=F(x_1)$. We will sometimes denote $G_{x_0x_1}$ simply by $G$ if there is no danger of confusion.  In this paper we consider the $(0,1)$-swapping of the inverse only, so, 
$G(x)=x^{p^n-2}+x^{p^n-1}+(x+1)^{p^n-1}$.

 We now deal with the binary swapping of the inverse function, and its $c$-differential  uniformity.  There is no need to consider $c=0$ for the $c$-differential uniformity, since we already know that $G$ is a permutation. Thus, in this paper we consider $c\neq 0,1$. We will be using the notation $i\mapsto {_c}\Delta_G(a,b)$ to mean a contribution of $i$ will be added to ${_c}\Delta_G(a,b)$.
 \begin{thm}
 Let $n\geq 2$ be a positive integer, $0,1\neq c\in\F_{2^n}$ and $F:\F_{2^n}\to\F_{2^n}$ be the inverse function defined by $F(x)=x^{2^n-2}$ and $G$ be its $(0,1)$-swapping. 
 If $n=2$,   then ${_c}\Delta_G(a,b)\leq 1$;  if $n=3$,    ${_c}\Delta_G(a,b)\leq 3$.
 If $n\geq 4$, and 
 for any $a,b\in\F_{p^n}$,  the $c$-DDT entries satisfy ${_c}\Delta_G(a,b)\leq 4$ \textup{(}all $[1,2,3,4]$ $c$-DDT entries occur\textup{)}. Furthermore, ${_c}\Delta_G(a,b)=4$ \textup{(}so, the $c$-differential uniformity of $G$ is $\delta_{F,c}=4$\textup{)} if and only if any of the conditions happen:
\begin{enumerate}
\item[$(i)$] For $a\in\F_{2^n}^*$ with $\Tr\left(\frac{a}{a+1} \right)=0$,  $b=\frac{1}{a+1}$ and $c=\frac{1}{a^2+a}$,  then ${_c}\Delta_G(a,\frac{1}{a+1})=4$;
\item[$(ii)$]  For $a\in\F_{2^n}^*$ with   $\Tr\left(\frac{a}{(a+1)^2} \right)=0$, $b=\frac{1}{a^2}$, and $c=\frac{a+1}{a^2}$, then   $ {_c}\Delta_G(a,\frac{1}{a^2})=4$. 
  \end{enumerate}
 \end{thm}
 \begin{proof}
  We display first the $c$-differential equation
\begin{equation}
\begin{split}
\label{eq:cdiff}
& (x+a)^{2^n-2}+(x+a)^{2^n-1}+(x+a+1)^{2^n-1}\\
&\qquad\qquad\qquad\qquad\qquad\qquad  + c \left( x^{2^n-2}+x^{p^n-1}+(x+1)^{2^n-1} \right)=b.
\end{split}
\end{equation}

 We ran a simple SageMath code and found easily the claims about $n=2,3,4$, so we now assume $n\geq 5$.
 
 For $a,b\in\F_{2^n}$, we  consider the $c$-differential equation~\eqref{eq:cdiff}. If $a=0$, then~\eqref{eq:cdiff} becomes
 \begin{align*}
 b&=x^{2^n-2}+x^{2^n-1}+(x+1)^{2^n-1}+c \left(x^{2^n-2}+x^{2^n-1}+(x+1)^{2^n-1} \right)\\
 &=(1+c) G(x).
 \end{align*}
 Since $G$ is also a permutation, then there exists a unique solution $x$, regardless of what $c\neq 1,b\in\F_{2^n}$. Thus, ${_c}\Delta_G(0,b)=1$. From here on forward, we will assume that $a\neq 0$.

 \vspace{.2cm}
 \noindent
 {\em Case $1$.} Let $a=1,b=0$.
 Then~\eqref{eq:cdiff} becomes
 \begin{align*}
 (x+1)^{2^n-2}+(x+1)^{2^n-1}+x^{2^n-1}+c \left(x^{2^n-2}+x^{2^n-1}+(x+1)^{2^n-1} \right) =0.
 \end{align*}
 Surely, $x=0$ is a solution if and only if $1+1+0+c(0+0+1)=0$, that is, $c=0$, a contradiction.   If $x=1$, then $0+0+1+c(1+1+0)=b$, so $b=1$, a contradiction.   If $x\neq 0,1$, then multiplying the displayed equation by $x(x+1)$, renders $x+c(x+1)=0$, and so, $x=\frac{c}{c+1}$, which implies $1\mapsto {_c}\Delta_G(1,0)$.

 \vspace{.2cm}
 \noindent
 {\em Case $2$.} Let $a=1,b=1$. 
 If $x=0$ in~\eqref{eq:cdiff}, then we must have $1+1+0+c(0+0+1)=1$, so $c=1$, a contradiction.
 If $x=1$ in~\eqref{eq:cdiff}, then we must have  $0+0+1+c(1+1+0)=1$, which is always true.
 If $x\neq 0,1$, then multiplying the displayed equation by $x(x+1)$, we get
 \[
 x+c(x+1)=x(x+1), \text{ that is, } x^2+c x+c=0.
 \]
 By Lemma~\ref{lem10}, since $c\neq 0$, two solutions (surely, $x=1$ cannot be one such solution) exist if and only if $\Tr(c^2)=\Tr(c)=0$,  therefore, we have   $3\mapsto {_c}\Delta_G(1,1)$ (including the prior $x=1$), under this condition, and a contribution  of $2$, otherwise. 

 \vspace{.2cm}
 \noindent
 {\em Case $3$.} Let $a=1,b\neq 0,1$.  If $x=0$, then we must have $1+1+0+c(0+0+1)=b$, that is, $b=c$. Thus, $1\mapsto {_c}\Delta_G(1,c)$. If $x=1$, then we must have $0+0+1+c(1+1+0)=b$, and so, $b=1$, a contradiction. If $x\neq 0,1$, multiplying Equation~\eqref{eq:cdiff} by $x(x+1)$ gives us
 \[
  x+c(x+1)=bx(x+1), \text{ that is, } x^2+\frac{b+c+1}{b} x+\frac{c}{b}=0,
 \]
 which, by Lemma~\ref{lem10}, under $b+c+1\neq 0$, has two distinct solutions if and only if $\Tr\left(\frac{bc}{(b+c+1)^2} \right)=0$. Thus, if $b=c$, and  $\Tr(c)=0$, then $3\mapsto {_c}\Delta_G(1,c)$, and if $b=c, \Tr(c)=1$, then $1\mapsto {_c}\Delta_G(1,c)$. If $b\neq c$, and $\Tr\left(\frac{bc}{(b+c+1)^2} \right)=0$, then $2\mapsto {_c}\Delta_G(1,b)$.
 
  \vspace{.2cm}
 \noindent
 {\em Case $4$.} Let $a\neq 0,1$.  
 If $x=0$, then we must have $\frac{1}{a}+1+1+c(0+0+1)=b$, so $b=c+\frac{1}{a}$. Thus, $1\mapsto {_c}\Delta_G(a,c+\frac{1}{a})$.
 If $x=1$, then $\frac{1}{a+1}+1+1+c(1+1+0)=b$, so $b=\frac{1}{a+1}$. Hence, $1\mapsto {_c}\Delta_G(a,\frac{1}{a+1})$.
 If $x=a$, then $0+0+1+c\left(\frac{1}{a}+1+1\right)=b$, so $b=1+\frac{c}{a}$. Therefore, $1\mapsto {_c}\Delta_G(a,1+\frac{c}{a})$.
 If $x=a+1$, then $1+1+0+c\left(\frac{1}{a+1}+1+1\right)=b$, so $b=\frac{c}{a+1}$. Thus, $1\mapsto {_c}\Delta_G(a,\frac{c}{a+1})$.
 Assuming $x\neq 0,1,a,a+1$, multiplying Equation~\eqref{eq:cdiff} by $x(x+a)$  implies
 \[
 x+c(x+a)=bx(x+a),\text{ that is, } x^2+\frac{ab+c+1}{b} x+\frac{ac}{b}=0.
 \]
 By Lemma~\ref{lem10}, under $ab+c+1\neq 0$, this last equation has two distinct solutions if and only if $\Tr\left(\frac{abc}{(ab+c+1)^2} \right)=0$. 
 
 We now summarize the largest contributions to ${_c}\Delta_G(a,b)$: if $b=c+\frac{1}{a}=\frac{1}{a+1}$ (so, $c=\frac{1}{a^2+a}$; observe that $Tr(\frac{1}{c})=0$) and  $\Tr\left(\frac{abc}{(ab+c+1)^2} \right)=\Tr\left(\frac{a^2}{(a+1)^2} \right)=Tr\left(\frac{a}{a+1} \right)=0$, then $4\mapsto {_c}\Delta_G(a,\frac{1}{a+1})$; if $b=c+\frac{1}{a}=\frac{c}{a+1}$ (so, $c=\frac{a+1}{a^2}$ and $b=\frac{1}{a^2}$; observe that $\Tr(c)=0$) and $\Tr\left(\frac{a}{(a+1)^2} \right)=0$, then $4\mapsto {_c}\Delta_G(a,\frac{1}{a^2})$. All values $[1,2,3,4]$ as $c$-DDT entries will occur, and we argue that next. For the values $1,2,3$, that is obvious from our case splitting.  The value $4$ will occur if we can find parameters satisfying the previous conditions, and we achieve that in the following way. We can  take $\beta\in\F_{2^n}^*$ such that $\Tr(\beta)=\Tr(1)$ and set $a:=\frac{1}{\beta}+1$, therefore, $\Tr\left(\frac{a}{a+1} \right)=\Tr(1)+\Tr\left(\frac{1}{a+1} \right)=\Tr(1)+\Tr(\beta)=0$. Thus, the second condition is satisfied, and therefore $\delta_{F,c}\leq 4$ with the upper bound being attained.
 \end{proof}

\section{The $c$-boomerang uniformity of the swapped inverse function}
\label{sec5}

 We now investigate the $c$-BCT for the same swapped inverse function. It was shown in~\cite{Li19} that if $c=1$, then the boomerang uniformity of $G_{01}$ is $10$, if $n\equiv 0\pmod 6$, $8$,  if $n\equiv 3\pmod 6$, and 6,  if $n\not\equiv 0\pmod 3$. We will be using the notation $i\mapsto {_c}\cB_G(a,b)$ to mean a contribution of $i$ will be added to ${_c}\cB_G(a,b)$.
 
 We will be using throughout the  $c$-boomerang system for the $(0,1)$-swapping of the binary inverse function:
\begin{equation}
\begin{split}
\label{eq:cboom}
 & (x+\gamma)^{p^n-2}+(x+\gamma)^{p^n-1}+(x+\gamma+1)^{p^n-1}\\
 &\qquad\qquad\qquad -c \left( x^{p^n-2}+x^{p^n-1}+(x+1)^{p^n-1} \right)=b,\\
& (x+\gamma+a)^{p^n-2}+(x+\gamma+a)^{p^n-1}+(x+\gamma+a+1)^{p^n-1}\\
 &\qquad\qquad -c^{-1}\left( (x+a)^{p^n-2}+(x+a)^{p^n-1}+(x+a+1)^{p^n-1} \right)=b. 
\end{split}
\end{equation}

 \begin{thm}
 \label{thmBCT}
 Let $n\geq 2$ be a positive integer, $0,1\neq c\in\F_{2^n}$ and $F:\F_{2^n}\to\F_{2^n}$ be the inverse function defined by $F(x)=x^{2^n-2}$ and $G$ be its $(0,1)$-swapping. 
 If $n=2$,  ${_c}\cB_G (a,b)=1$;  if $n=3$, for $a,b\in\F_{2^n}$, ${_c}\cB_G (a,b)\leq 4$ \textup{(}all $[1,2,3,4]$ entries  occur\textup{)}.
 If $n\geq 4$, and 
 for any $a,b\in\F_{2^n}$,  the $c$-BCT entries satisfy ${_c}\cB_G(a,b)\leq 5$. The upper bound is attained (at least) for $n\equiv 0\pmod 3$, $n\equiv 0\pmod 4$, $n\equiv 0\pmod 5$. 
 \end{thm}
 \begin{proof}
 The instances $n=2,3$ were treated  with a SageMath program. We now assume that $n\geq 4$. We do not need to consider the cases of $a=0$, since, then ${_c}\cB_G (0,b)={_c}\Delta_G(0,b)$, which was considered earlier. 
 
 Let $a=1$. The $c$-boomerang system~\eqref{eq:cboom} becomes
 \begin{equation}
 \begin{split}
 \label{eq:a1}
 & (x+\gamma)^{2^n-2}+(x+\gamma)^{2^n-1}\\
 &\qquad\qquad +(x+\gamma+1)^{2^n-1}+c \left( x^{p^n-2}+x^{p^n-1}+(x+1)^{p^n-1} \right)=b,\\
 & (x+\gamma+1)^{2^n-2}+(x+\gamma+1)^{2^n-1}+(x+\gamma)^{2^n-1}\\
 &\qquad\qquad +\frac{1}{c} \left( (x+1)^{p^n-2}+x^{p^n-1}+(x+1)^{p^n-1} \right)=b.
 \end{split}
 \end{equation}
 If $x=0$, then $b=\gamma^{2^n-2}+\gamma^{2^n-1}+(\gamma+1)^{2^n-1}+c=(\gamma+1)^{2^n-2}+\gamma^{2^n-1}+(\gamma+1)^{2^n-1}$. If $\gamma=0,1$, then $c=1$, a contradiction. If $\gamma\neq 0,1$, then $b=\frac{c}{c+1}$ and  $ {1\mapsto {_c}\cB_G\left(1,\frac{c}{c+1}\right)}$. 
 
 Next, if $x=1$, then $b=(\gamma+1)^{2^n-2}+(\gamma+1)^{2^n-1}+\gamma^{2^n-1}=\gamma^{2^n-2}+\gamma^{2^n-1}+(\gamma+1)^{2^n-1}+\frac{1}{c}$. If $\gamma=0,1$ we get a contradiction. If $\gamma\neq 0,1$, then $b=\frac{1}{\gamma+1}=\frac{1}{\gamma}+\frac{1}{c}$, rendering $b=\frac{1}{c+1}$ and so, ${1\mapsto {_c}\cB_G\left(1,\frac{1}{c+1}\right)}$.
 
 If $x=\gamma$, then $b=1+\frac{c}{x}=\frac1{c(x+1)}$, so
 $c b^2  + (c^2+c+1)b+1=0$. 
 There exists a solution $b=c^{1/2}$ if $c^2+c+1=0$ and so, ${1\mapsto {_c}\cB_G\left(1,c^{1/2}\right)}$, or two solutions $b$ if $c^2+c+1\neq 0$ and $\Tr\left(\frac{c}{(c^2+c+1)^2 }\right)=0$, and therefore ${1\mapsto {_c}\cB_G\left(1,b\right)}$, for any such $b$. 
 
 Similarly, if $x=\gamma+1$, the obtained equation is $cb^2+(c^2+c+1)b+c^2=0$.
 If $c^2+c+1=0$, then $b=c^{1/2}$ and ${1\mapsto {_c}\cB_G\left(1,c^{1/2}\right)}$, and  if $c^2+c+1\neq 0$ and $\Tr\left(\frac{c^3}{(c^2+c+1)^2 }\right)=0$, and therefore ${1\mapsto {_c}\cB_G\left(1,b\right)}$, for any such $b$. 
 
 If $x\neq 0,1,\gamma,\gamma+1$, the system is then $b=\frac1{x+\gamma}+\frac{c}{x}
= \frac{1}{x+\gamma+1}+\frac1{c(x+1)}$, implying that $b^2 c x^2+ (1 + c + b) (1 + c + b c) x+c(1+ c + b c) =0$. If $b=c+1$, or $b=\frac{c+1}{c}$, there is a unique solution $b$, and if $c+1\neq b\neq \frac{c+1}{c}$ and $\Tr\left(\frac{b^2c^2}{(1+c+bc)(1+c+b)^2} \right)$, then ${1\mapsto {_c}\cB_G\left(1,b\right)}$, for any such $b$.

 We quickly see that the maximum claimed $c$-BCT entry cannot be achieved when $a=1$. 
  
 From here on forward, we take $a\neq 0,1$.
We will consider the  cases of $(x,\gamma)$ being one of: 
\begin{align*}
&(0,0), (1,1), (0,1), (1,0), (0,a)_{a\neq 0,1}, (a,0)_{a\neq 0,1}, (0,a+1)_{a\neq 0,1}, \\
& (a+1,0)_{a\neq 0,1},  (1,a)_{a\neq 0,1},(a,1)_{a\neq 0,1},(1,a+1)_{a\neq 0,1}, (a+1,1)_{a\neq 0,1}, \\
&  (a,a)_{a\neq 0,1}, (a+1,a+1)_{a\neq 0,1},(a,a+1)_{a\neq 0,1},  (a+1,a)_{a\neq 0,1}, \\
& (x,x)_{x\neq 0,1,a,a+1}, (x,x+1)_{x\neq 0,1,a,a+1},
 (0,\gamma)_{\gamma\neq 0,1,a,a+1},(1,\gamma)_{\gamma\neq 0,1,a,a+1},\\
&(a,\gamma)_{\gamma\neq 0,1,a,a+1}, (a+1,\gamma)_{\gamma\neq 0,1,a,a+1},
 (x,x+a)_{x\neq 0,1,a,a+1}, \\
 & (x,x+a+1)_{x\neq 0,1,a,a+1}, x+\gamma\neq 0,1,a,a+1, x\neq 0,1,a,a+1.
\end{align*}

  \vspace{.2cm}
 \noindent
 {\em Case $1$.} Let $x=\gamma=0$. The $c$-boomerang system becomes (recall that $a\neq 0,1$), $ b=c+1= \frac{1}{a}+  \frac{1}{ac}=\frac{c+1}{ac}$. Thus, $a=\frac{1}{c},b=c+1$, so ${1\mapsto {_c}\cB_G\left(\frac{1}{c},c+1\right)}$.
 
   \vspace{.2cm}
 \noindent
 {\em Case $2$.}  Let $x=\gamma=1$. The $c$-boomerang system becomes transforms into
$b=1=\frac{1}{a}+\frac{1}{c(a+1)}$, implying $a(a+1)c+c(a+1)+a=0$, that is, 
 \begin{equation}
 \label{eq0}
 a^2c+a+c=0,
 \end{equation}
  which, by Lemma~\ref{lem10} has solutions $a$ if and only $\Tr\left(c^2\right)=\Tr\left(c\right)=0$ (using Hilbert's Theorem~90), so 
 ${1\mapsto {_c}\cB_G\left(a,1\right)}$, for such an~$a$.

    \vspace{.2cm}
 \noindent
 {\em Case $3$.}  Let $(x,\gamma)=(0,1)$. Then, we must have   $b=c=\frac{1}{a+1}+\frac{1}{ac}$, so, 
 \begin{equation}
 \label{eq1}
 a^2c^2+a(c^2+c+1)+1=0. 
 \end{equation}
 If $c^2+c+1=0$,
  then $a=\frac{1}{c}$, so, ${1\mapsto {_c}\cB_G\left(\frac{1}{c},c\right)}$, and if $c^2+c+1\neq 0$ and $\Tr\left(\frac{c}{c^2+c+1} \right)=0$, then ${1\mapsto {_c}\cB_G\left(a,c\right)}$.
 
     \vspace{.2cm}
 \noindent
 {\em Case $4$.}  Let $(x,\gamma)=(1,0)$. Then, we must have   $b=0=\frac{1}{a+1}+\frac{1}{c(a+1)}$, so $c=1$, which is not allowed.

    \vspace{.2cm}
 \noindent
 {\em Case $5$.}  Let $(x,\gamma)=(0,a)$, $a\neq 0,1$. We must have $b=\frac{1}{a}+c=1+\frac{1}{ca}
$, so $a=\frac{1}{c}, b=0$ and  ${1\mapsto {_c}\cB_G\left(\frac{1}{c},0\right)}$.

 \vspace{.2cm}
 \noindent
 {\em Case $6$.}  Let $(x,\gamma)=(a,0)$, $a\neq 0,1$. Then, $b=\frac{1}{a}+\frac{c}{a}=1+\frac{1}{c}$, so $a=c,b=1+\frac{1}{c}$, and ${1\mapsto {_c}\cB_G\left(c,1+\frac{1}{c}\right)}$.
 
  \vspace{.2cm}
 \noindent
 {\em Case $7$.}  Let $(x,\gamma)=(0,a+1)$, $a\neq 0,1$. 
 Then, $b=\frac{1}{a+1}+c=\frac{1}{ac}$, so Equation~\eqref{eq1} holds.  If $c^2+c+1=0$,
  then $a=\frac{1}{c},b=1$, so, ${1\mapsto {_c}\cB_G\left(\frac{1}{c},1\right)}$, and if $c^2+c+1\neq 0$ and $\Tr\left(\frac{c}{c^2+c+1} \right)=0$, then ${1\mapsto {_c}\cB_G\left(a,\frac{1}{ac}\right)}$. 
 
   \vspace{.2cm}
 \noindent
 {\em Case $8$.}  Let $(x,\gamma)=(a+1,0)$, $a\neq 0,1$. The $c$-boomerang system is then $b=\frac{1}{a+1}+\frac{c}{a+1}=0$, so $c=1$, a contradiction.

     \vspace{.2cm}
 \noindent
 {\em Case $9$.}  Let $(x,\gamma)=(1,a)$, $a\neq 0,1$. We must have $b=\frac{1}{a+1}=\frac{1}{c(a+1)}$, so $c=1$, a contradiction.

  \vspace{.2cm}
 \noindent
 {\em Case $10$.}  Let $(x,\gamma)=(a,1)$, $a\neq 0,1$. Then, $b=\frac{1}{a+1}+\frac{c}{a}=\frac{1}{c}$, so 
 \begin{equation*}
 \label{eq2}
 a^2 + (c^2 + c + 1) a+ c^2 =0.
 \end{equation*} 
 If $c^2+c+1=0$,
  ${1\mapsto {_c}\cB_G\left(c,\frac{1}{c}\right)}$, and if $c^2+c+1\neq 0$ and $\Tr\left(\frac{c}{c^2+c+1} \right)=0$, then ${1\mapsto {_c}\cB_G\left(a,\frac{1}{c}\right)}$.

  \vspace{.2cm}
 \noindent
 {\em Case $11$.}  Let $(x,\gamma)=(1,a+1)$, $a\neq 0,1$. Thus, $b=\frac{1}{a}=1+\frac{1}{c(a+1)}$, implying Equation~\eqref{eq0}. Thus, if $\Tr(c)=0$, then ${1\mapsto {_c}\cB_G\left(a,\frac{1}{a}\right)}$.

   \vspace{.2cm}
 \noindent
 {\em Case $12$.}  Let $(x,\gamma)=(a+1,1)$, $a\neq 0,1$. Then, $b=\frac{1}{a}+\frac{c}{a+1}=1$, so
 \begin{equation*}
 \label{eq3}
 a^2+ac+1=0.
 \end{equation*}
 If $\Tr\left(\frac{1}{c} \right)=0$, then ${1\mapsto {_c}\cB_G\left(a,1\right)}$.

      \vspace{.2cm}
 \noindent
 {\em Case $13$.}  Let $(x,\gamma)=(a,a)$, $a\neq 0,1$. Then, we must have $b=1+\frac{c}{a}=\frac{1}{a}+\frac{1}{c}$, so $a=c,b=0$. Thus, ${1\mapsto {_c}\cB_G\left(c,0\right)}$.

  \vspace{.2cm}
 \noindent
 {\em Case $14$.}  Let $(x,\gamma)=(a+1,a+1)$, $a\neq 0,1$. Then, $b=1+\frac{c}{a+1}=\frac{1}{a}$, that is, 
 \[
 a^2+ac+1=0.
 \]
  Therefore, if $\Tr\left(\frac{1}{c} \right)=0$,  ${1\mapsto {_c}\cB_G\left(a,\frac{1}{a}\right)}$.

   \vspace{.2cm}
 \noindent
 {\em Case $15$.}  Let $(x,\gamma)=(a,a+1)$, $a\neq 0,1$. Then, we must have
 $b=\frac{c}{a}=\frac{1}{a+1}+\frac{1}{c}$, and so, 
$
 a^2 + (c^2 + c + 1) a+ c^2 =0.
$
  If $c^2+c+1=0$,
  ${1\mapsto {_c}\cB_G(c,1)}$, and if $c^2+c+1\neq 0$ and $\Tr\left(\frac{c}{c^2+c+1} \right)=0$, then ${1\mapsto {_c}\cB_G\left(a,\frac{c}{a}\right)}$.

  \vspace{.2cm}
 \noindent
 {\em Case $16$.}  Let $(x,\gamma)=(a+1,a)$, $a\neq 0,1$. We must have $b=\frac{c}{a+1}=\frac{1}{a+1}$, that is, $c=1$, a contradiction.

   \vspace{.2cm}
 \noindent
 {\em Case $17$.}  Let $(x,\gamma)=(x,x)$, $x\neq 0,1,a,a+1$. Thus,   $b=1+\frac{c}{x}  =\frac{1}{a}+\frac{1}{c(x+a)}$, so 
\begin{equation}
\label{eq:bb2}
a^2 c\,b^2+b\,a(c^2+c+1+ac)+ a c+a+c^2 =0.
\end{equation}
 Equation~\eqref{eq:bb2} will have one  solution $b=\frac{1}{c^2+c+1}$ when $a=\frac{c^2+c+1}{c}$ (observe that $c^2+c+1\neq 0$, since otherwise, $c=0$),  and so, ${1\mapsto {_c}\cB_G\left(\frac{c^2+c+1}{c},\frac{1}{c^2+c+1}\right)}$, and two $b$'s if $\Tr\left(\frac{c \left(a c+a+c^2\right)}{\left(a   c+c^2+c+1\right)^2} \right)  = \Tr\Big(\frac{c }{a c+c^2+c+1} +\frac{c^2}{\left(a   c+c^2+c+1\right)^2}+\frac{c \left(a+1\right)}{\left(a   c+c^2+c+1\right)^2}\Big) = \Tr\left(\frac{c \left(a+1\right)}{\left(a   c+c^2+c+1\right)^2}\right)=0$, and,  for such  $a,b$, we have ${1\mapsto {_c}\cB_G(a,b)}$.

   \vspace{.2cm}
 \noindent
 {\em Case $18$.}  Let $(x,\gamma)=(x,x+1)$, $x\neq 0,1,a,a+1$. Then,   $b=\frac{c}{x}= \frac{1}{a+1}+\frac{1}{c(x+a)}$, so 
  \[
  a(a+1)c\, b^2  + b (1 + a + a c + c^2 + a c^2)+c^2=0,
  \]
 which has a unique solution $b=\frac{c^2+c+1}{c+1}$ if $a=\frac{c^2+1}{c^2+c+1}$,  and consequently, ${1\mapsto {_c}\cB_G\left(\frac{c^2+1}{c^2+c+1},\frac{c^2+c+1}{c+1}\right)}$, or two solutions $b$ if and only if $a\neq \frac{c^2+1}{c^2+c+1}$ and $\Tr\left(\frac{c^3 a(a+1)}{(1 + a + a c + c^2 + a c^2)^2} \right)=0$, when ${1\mapsto {_c}\cB_G(a,b)}$, for such $a,b$.
 
    \vspace{.2cm}
 \noindent
 {\em Case $19$.}  Let $(x,\gamma)=(0,\gamma)$, $\gamma\neq 0,1,a,a+1$. The system becomes   
  $b=\frac{1}{\gamma}+c=\frac{1}{\gamma+a}+ \frac{1}{ca}$, implying
  \begin{equation}
  \label{eq:bb4}
  a^2 c\, b^2 +b\,a(1+ac^2)+ac^2+ac+1=0.
  \end{equation}
  If $a=\frac{1}{c^2}$, then $b=c$ and ${1\mapsto {_c}\cB_G\left(\frac{1}{c^2},c\right)}$. If $a\neq \frac{1}{c^2}$ and $\Tr\left(\frac{a^2 c(ac^2+ac+1)}{(1+ac^2)^2} \right)=\Tr\left(\frac{a^2c}{1+ac^2} +\frac{a^4 c^2}{(1+ac^2)^2}+\frac{a^3(a+1)c^2}{(1+ac^2)^2}\right)=\Tr\left(\frac{a^3(a+1)c^2}{(1+ac^2)^2}\right)=0$, then  ${1\mapsto {_c}\cB_G\left(a,b\right)}$, for $a,b$ satisfying Equation~\eqref{eq:bb4}.
    
     \vspace{.2cm}
 \noindent
 {\em Case $20$.}  Let $(x,\gamma)=(1,\gamma)$, $\gamma\neq 0,1,a,a+1$. Thus,  
 $b=\frac{1}{\gamma+1}=\frac{1}{\gamma+a+1}+\frac{1}{c(a+1)}$, so 
 \[
 a (a+1)c\, b^2+ab+1=0.
 \]
  Thus, if $\Tr\left(\frac{(a+1)c}{a} \right)=0$, then ${1\mapsto {_c}\cB_G\left(a,b\right)}$.

      \vspace{.2cm}
 \noindent
 {\em Case $21$.}  Let $(x,\gamma)=(a,\gamma)$, $\gamma\neq 0,1,a,a+1$. The system becomes    $b=\frac1{\gamma+a}+\frac{c}{a}=\frac1{\gamma}+\frac{1}{c}$, so
\begin{equation}
\label{eq:bb6}
a^2 c\, b^2  + b \,a (a + c^2)+c^2+ a c+ a=0.
\end{equation}
If $a=c^2$, then $b=\frac{1}{c}$ and ${1\mapsto {_c}\cB_G\left(c^2,\frac{1}{c}\right)}$, and if $a\neq c^2$ and $\Tr\left(\frac{c(c^2+ac+a)}{(a+c^2)^2} \right)= \Tr\left(\frac{c}{a+c^2}+ \frac{a}{a+c^2}+\frac{a^2}{(a+c^2)^2}\right)=\Tr\left(\frac{c}{a+c^2}\right)=0$ (we used Hilbert's Theorem 90), then ${1\mapsto {_c}\cB_G\left(a,b\right)}$, for $a,b$ satisfying Equation~\eqref{eq:bb6}.

      \vspace{.2cm}
 \noindent
 {\em Case $22$.}  Let $(x,\gamma)=(a+1,\gamma)$, $\gamma\neq 0,1,a,a+1$. Then, 
   $b=\frac1{\gamma+a+1}+\frac{c}{a+1} =\frac1{\gamma+1}$, so 
\begin{equation}
\label{eq:bb7}
a(a+1)\, b^2 + ac\, b + c=0.
\end{equation}
Thus, if $\Tr\left(\frac{a+1}{ac} \right)=0$, then ${1\mapsto {_c}\cB_G\left(a,b\right)}$.

      \vspace{.2cm}
 \noindent
 {\em Case $23$.}  Let $(x,\gamma)=(x,x+a)$, $x\neq 0,1,a,a+1$, $a\neq 0,1$. 
 Thus, $b=\frac1{a}+ \frac{c}{x}=1+\frac{1}{c(x+a)}$, so 
 \begin{equation}
 \label{eq:bb8}
 a^2 c\, b^2 + b\,a (c^2 + c + 1 + ac)+1+ac+ac^2=0.
 \end{equation}
 If  $a=\frac{c^2+c+1}{c}$, then $b=\frac{c^2}{c^2+c+1}$, 
  and ${1\mapsto {_c}\cB_G\left(\frac{c^2+c+1}{c},\frac{c^2}{c^2+c+1}\right)}$. If $a\neq \frac{c^2+c+1}{c}$ and $\Tr\left( \frac{c(ac^2+ac+1)}{(c^2 + c + 1 + ac)^2}\right)=\Tr\left(\frac{c^2}{\left(a   c+c^2+c+1\right)^2}+\frac{c}{a   c+c^2+c+1}+\frac{(a+1) c^3}{\left(a   c+c^2+c+1\right)^2} \right)=\Tr\left( \frac{(a+1) c^3}{\left(a
   c+c^2+c+1\right)^2}\right)=0$, then ${1\mapsto {_c}\cB_G\left(a,b\right)}$.

     \vspace{.2cm}
 \noindent
 {\em Case $24$.}  Let $(x,\gamma)=(x,x+a+1)$, $x\neq 0,1,a,a+1$, $a\neq 1$.  Then,
 $b=\frac1{a+1}+\frac{c}{x}=\frac{1}{c(x+a)}$, so 
 \[
  b^2 a(a+1)c + b (1 + a + a c + c^2 + a c^2)+1=0.
 \]
 If $a=\frac{c^2+1}{c^2+c+1}$, $b=\frac{c^2+c+1}{c(c+1)}$, and 
  ${1\mapsto {_c}\cB_G\left(\frac{c^2+1}{c^2+c+1},\frac{c^2+c+1}{c(c+1)}\right)}$, and if 
 $a\neq \frac{c^2+1}{c^2+c+1}$ and $\Tr\left( \frac{a(a+1)c}{(1 + a + a c + c^2 + a c^2)^2}\right)=0$, then ${1\mapsto {_c}\cB_G\left(a,b\right)}$.

     \vspace{.2cm}
 \noindent
 {\em Case $25$.}  Let $x+\gamma\neq 0,1,a,a+1$, $x\neq 0,1,a,a+1$.   
 Then, we must have $b=\frac1{x+\gamma}+\frac{c}{x}=\frac{1}{x+\gamma+a}+\frac{1}{c(x+a)}$, which implies 
 \[
 b x^2 + c y + x (1 + c + b y)=0,   a + a c + a^2 b c + b c x^2 + (1 + a b c) y + x (1 + c + b c y)=0.
 \]
  Adding to the second equation, a $c$ multiple of the first gives us the system (the advantage is that it shows the linear relationship between $x,y$), and replacing $y$ into the second displayed equation, we get
 \begin{equation}
 \begin{split}
 \label{eq:xy}
 & a + a c + a^2 b c + (1 + c^2) x + (1 + a b c + c^2) y=0,\\
 &a b^2 c\, x^2+(1 + a b + c) (1 + c + a b c) x +  a c (1+c +abc) =0.
 \end{split}
 \end{equation}

If $b=\frac{c+1}{a}$, or $b=\frac{c+1}{ac}$, then the second  equation in~\eqref{eq:xy} has one solution $x$ and if $\frac{c+1}{a}\neq b\neq \frac{c+1}{ac}$ and $\Tr\left( \frac{a^2b^2c^2}{(1 + a b + c)^2 (1 + c + a b c)}\right)=0$ it has two solutions. In the first case, ${1\mapsto {_c}\cB_G\left(a,b\right)}$, and in the second case,  ${2\mapsto {_c}\cB_G\left(a,b\right)}$ (for the same $(a,b,c)$ tuple, each of the solutions $x$ will render a solution $y$: that is the reason we wrote one of the equations as an affine plane).

   Now, we need to see how the multiple $(a,b)$ pairs (with or without conditions) overlap, as the contributions to the various $c$-BCT entries will be added. For easy referral, we shall combine them in a set where we put an asterisk if there are conditions on $(a,b)$ (if the conditions are different, then we repeat that pair):
{\small  \begin{align*}
  S&=\Bigg\{
  \left(\frac{1}{c},c+1\right), (a,1)^*,    \left(\frac{1}{c},c\right)^*, (a,c)^*,  (c,0),    \left(a,\frac{1}{a}\right)^*,   (c,1)^*, \left(a,\frac{c}{a}\right)^*,    \\
  & \left(\frac{1}{c},0\right), \left(c,1+\frac{1}{c}\right), \left(\frac{1}{c},1\right)^*, \left(a,\frac{1}{ac}\right)^*, (a,c) ^*, \left(c,\frac{1}{c}\right)^*, \left(a,\frac{1}{c}\right)^*, \left(a,\frac{1}{a}\right)^*,  \\
  & (a,1)^*, \left(\frac{c^2+c+1}{c},\frac1{c^2+c+1} \right), (a,b)^*,  \left(\frac{c^2+1}{c^2+c+1},\frac{c^2+c+1}{c+1} \right),   \\
  &(a,b)^*,   \left(\frac{1}{c^2},c\right), (a,b)^*, (a,b)^*,  \left(c^2,\frac{1}{c}\right), (a,b)^*,  (a,b)^*, \\
  & \left(\frac{c^2+c+1}{c},\frac{c^2}{c^2+c+1} \right),  (a,b)^*, \left(\frac{c^2+1}{c^2+c+1},\frac{c^2+c+1}{c^2+c} \right), (a,b)^*, (a,b)^*
   \Bigg\}.
   \end{align*}
   }
   We will prune the above set, concentrating on the maximal entry value $ {_c}\cB_G(a,b)$, which we claim to be $5$.  Going  through every possibility (we shall go through similar arguments enabling us to motivate the maximum $c$-BCT entry below, so we skip the previous  straightforward pruning argument, albeit tedious), we obtained that the only possible pairs that could possibly overlap five times (we write the constraints next to each of them) are in the displayed list below (recall that $p_{10}$ contributes twice to ${_c}\cB_G(a,b)$).    
 Since we shall refer to the respective pairs in the arguments below the list, by relabeling, we will call the pairs $(a,b)$ satisfying the conditions from which $E_i$ are derived, as coming from Case $C_i$. Also, if the pairs $p_i,p_j$ match the conditions from cases $C_i,C_j$,  we write $p_i=p_j$. We bring up an issue the reader should be aware of, which we alluded to previously: if the Case $C_i$ holds, then $E_i=0$; however, if $E_i=0$, the pair $(a,b)$ may or may not be $p_i$, so one has to continuously check. We find that to be the difficulty in the analysis below, along with the many cases.


{\small
   \allowdisplaybreaks
  \begin{align*}
  &C_0: p_0=\left(a,\frac{1}{ac}\right):   E_0= a^2 c^2 + (c^2+c+1) a + 1=0;\\
 &C_1: p_1=\left(a,\frac{c}{a}\right):  E_1=a^2+(c^2+c+1)a+c^2=0;\\
 &C_2: p_2= (a,b):   E_2=a^2 c\,b^2+a(c^2+c+1+ac)\, b+ a c+a+c^2=0; \\
 & C_3: p_3=(a,b):  E_3= a(a+1)c\, b^2+ b (1 + a + a c + c^2 + a c^2)+c^2=0; \\
& C_4:   p_4=(a,b):  E_4=a^2 c\, b^2 +a(1+ac^2)\, b+ac^2+ac+1=0;\\
& C_5: p_5=(a,b):   E_5=a (a+1)c\, b^2+ab+1=0; \\
&C_6:  p_6=(a,b):    E_6=a^2 c\, b^2  + b \,a (a + c^2)+c^2+ a c+ a=0;\\
& C_7: p_7=(a,b):    E_7=  a(a+1)b^2+acb+c=0; \\
&C_8: p_8=(a,b):   E_8=  a^2 c\, b^2 + a (1 + c + c^2 + ac)b+1+ac+ac^2=0;\\
&C_9:  p_9=(a,b):    E_9=  a(a+1)cb^2+b(1+a+ac+c^2+ac^2)+1=0; \\
&C_{10}:  p_{10}=(a,b):    E_{10}=   a b^2 c\, x^2+(1 + a b + c) (1 + c + a b c) x +  a c (1+c +abc) =0, \\
& \qquad\qquad \text{as well as } E_{10}'=a + a c + a^2 b c + (1 + c^2) x + (1 + a b c + c^2) y=0; \\
   \end{align*}
   }

We make the following  important observation: there is a duality between all of the $E_i$, $0\leq i\leq 9$, expressions. By duality we mean that  one expression is obtained from the other by replacing $c$ by $\frac{1}{c}$. Precisely, $E_0\longleftrightarrow E_1$, $E_2 \longleftrightarrow E_8$, $E_3\longleftrightarrow E_9$, $E_4\longleftrightarrow E_6$, $E_5\longleftrightarrow E_7$.

And here is a simple argument that shows that in each dual pair, if one case occurs in a solution set, then the other cannot, which shows that the $c$-boomerang uniformity is upper bounded by $5$ (we exclude two cases out of the possible $6$ (recall that $C_{10}$ contributes 2 to ${_c}\cB_G(a,b)$) by showing that $\{p_0,p_1\}\cap \{p_5,p_7\}$, as well as $\{p_4,p_6\}\cap \{p_2,p_8\}$,  if nonempty, cannot be part of quintuple  solution sets).

If $p_0=p_1$, then $b=\frac{c}{a}=\frac{1}{ac}$, implying $c=1$, which is impossible.
If $p_2=p_8$, then $1 + a^2 b^2 c + a (b + b^2 c)=0$, $a^2 b^2 + c + a (b^2 + b c)=0$, 
If $p_3=p_9$, then adding $E_3+E_9=0$ we get $c^2+1$, which is impossible.
If $p_4=p_6$, then adding $(1 + b c) E_4 + c (b + c) E_6=0$, renders $b=\frac{b c+c^2+1}{b^2 c+b c^2+b c+b+c}$, which when used in 
 $b E_4 + c E_6=0$, finds $b=\frac{c^2+c+1}{c}=0$ and $a=1$, which, when 
 plugged back into $p_4,p_6$ imply $\frac{c \gamma_4+c+\gamma_4}{c \gamma_4}=\frac{c   \gamma_4+c+\gamma_4}{\gamma_4+1}$ and $\frac{c \gamma_6+\gamma_6+1}{\gamma_6}=\frac{c   \gamma_6+\gamma_6+1}{c (\gamma_6+1)}$, which gives $c=1$, an impossibility.
 If $p_5=p_7$, then  adding $ E_5 + c E_7=0$, we get $b=\frac{1}{a}$, which, when replaced back, gives $c(1+1/a)=0=1+1/a$, and so, $a=1$. However,  $a=1,b=1$, contradicts the conditions imposed in  $C_5,C_7$ (Cases $20$ and $21$), since then $\gamma=0$.
 Further, if $p_0=p_5$, then $b=\frac{1}{ac}$, which used in $E_5$ gives $a=\frac{1}{c},b=1$ (and $c^2+c+1=0$). These values used in ${E_2, E_3, E_4, E_6, E_8, E_9}$ gives us
\[
c + c^2, 1 + c, 1 + c, \frac{1}{c}, \frac{1 + c}{c}, c + c^2,
\] 
and none can be zero. If $p_0=p_7$, $b=\frac{1}{ac}$, then $a=\frac{1+c^2}{c}, b=\frac1{c^2+1}$ (and $ c^5+c^4+c^3+c+1=0$). These values used in $E_2, E_3, E_4, E_6, E_8, E_9$ gives us
\[
c, c^2(c^2+c+1), 1 + c + c^3, \frac1{c^2}, c(c^3+1), \frac{c(c^2+c+1)}{(c+1)^2}.
\]
If $c^2+c+1=0$, then $ c^5+c^4+c^3+c+1=0$ implies $c=1$.  If $c^3+c+1=0$, then $c=0,1$, which is not allowed.
Next, we show that $p_1$ cannot be part of a quintuple solution along with any of the $\{p_5,p_7\}$. First, we assume that $p_1=p_5$ and so, $b=\frac{c}{a}$ and $a=\frac{c^3}{c^3+c+1}$. When plugged back into $E_1=0$, it implies $c^5+c^4+c^2+c+1=0$. Now, these values used in $E_2, E_3, E_4, E_6, E_8, E_9$ gives us
\[
\frac{c+1}{c^5+1}, \frac{1 + c}{c^2}, \frac{c}{1 + c + c^3}, \frac{ c^2(c^3+1)}{1 + c + c^3}, c, \frac{(1 + c)(1 + c^2 + c^3)}{c^2}.
\]
Surely, if $c^3+1=0$, then $c^5+c^4+c^3+c+1=c^4(c+1)\neq 0$, and therefore,  we cannot make up a quintuple solution.
If $p_1=p_7$, then $b=\frac{c}{a}$, so $a=c$, and $b=1$. Replacing these values into  $E_2, E_3, E_4, E_6, E_8, E_9$  gives us
\[
c^2 (c+1), c+1 , (c+1)(c^3+c^2+1) , c(c+1) ,  c+1, c(c+1),
\]
but none of these values can be zero. 

We move on to the second dual pairs $\{p_4,p_6\}\cap \{p_2,p_8\}$. If $p_2=p_4$, then adding $E_2+E_4$ we get $b=\frac{1+c}{ac}$, which when put back into $E_2=0$, say, gives us $c=0$, which is not allowed. If $p_2=p_6$, then adding $E_2+E_6$ implies $ab(a+1)(c+1)=0$. If $a=0$, then $E_2=c^2$, which is not permitted. If $b=0$, then $E_2=a+ac+c^2=0$, so $a=\frac{c^2}{c+1}$. Putting these values into $\{E_1,E_3,E_5,E_7,E_9\}$ renders
\[
\left\{\frac{c^5}{(c+1)^2},c^2,1,c,1\right\},
\]
and none of these values can be zero, so no ``winning'' quintuple.
If $a=1$, then $b=\frac{c^2+c+1}{c}$, which when used into $\{E_1,E_3,E_5,E_7,E_9\}$ gives
\[
\left\{c,c+1,\frac{(c+1)^2}{c},(c+1)^2,c (c+1)\right\},
\]
which cannot be zero.
If $p_4=p_8$, adding $E_4+E_8=0$, we get $a bc(a+1) (c+1 )$. If $a=0$, then $E_4=1$, which is not possible. If $b=0$, then $a=\frac{1}{c^2+c}$, which, when used in $\{E_1,E_3,E_5,E_7,E_9\}$ gives
\[
\left\{\frac{c^6+c+1}{c^2   (c+1)^2},c^2,1,c,1\right\},
\]
so no possible quintuple. If $a=1$, then $b=\frac{c^2+c+1}{c}$, and $\{E_1,E_3,E_5,E_7,E_9\}$ become
\[
\left\{c^2,\frac{c^4+c+1}{c},1,c,\frac{c^4+c^3
   +1}{c}\right\}.
\]
We could possibly get a solution quintuple if $1 + c + c^4=0$ and $1 + c^3 + c^4=0$, but that is impossible.
If $p_6=p_8$, adding $E_6+E_8=0$, we get $b=\frac{1+c}{a}$, but with this value both $E_6,E_8$ become equal to $c$ and consequently not zero.

The above discussion shows that ${_c}\cB_G(a,b)\leq 5$.
Since the rest of  the proof (which is long enough, as it is) deals with finding appropriate cases where the upper bound is attained we moved it to Appendix~A. The theorem is shown.
\end{proof}

We make the following conjecture, based upon our extensive computations.
\begin{conjecture}
For all  $n\geq 4$, then there exists $c\neq 0,1$ such that $\beta_{F,c}=5$.
\end{conjecture}
   
\section{Concluding remarks}
\label{sec6}

In this paper we investigate the $c$-differential and $c$-boomerang uniformity of the swapped inverse function. We show that its $c$-differential uniformity is less than or equal to $4$ and its $c$-boomerang uniformity is upper bounded by $5$.
As we saw, investigating questions on $c$-differential and $c$-boomerang uniformities tends to be quite difficult, but it sure is interesting to find other classical PN/APN functions and  study their properties through the new differential. It will also be worthwhile to check into the general $p$-ary  versions of the results from this paper.

\appendix
\section*{Appendices}
\addcontentsline{toc}{section}{Appendices}
\renewcommand{\thesubsection}{\Alph{subsection}}
 \subsection{Completing the proof of Theorem~\ref{thmBCT}}
 
 Here, we finish the proof of Theorem~\ref{thmBCT}.
 We cannot find a simple argument to show that $5$ is always attained {\em for all dimensions}, but we can find some values for $c$ such that  ${_c}\cB_G(a,b)=5$, namely, for $n\equiv 0\pmod 3$, $n\equiv 0\pmod 4$, $n\equiv 0\pmod 5$. We start by showing how we ``guessed'' a quintuple of cases that provide a solution set for the $c$-boomerang system. Recall that if the pair $p_i$ happens then $E_i$ will hold, but not necessarily the other way around.
 
 If $p_0=p_2$ (no need to consider $p_0=p_1$ as they are dual), then $b=\frac{1}{ac}$, and so, $a=\frac{c^2+c+1}{c}, b=\frac1{c^2+c+1}$, along with $c^5+c^4+c^3+c^2+1=0$. Using these values into $\{E_1,E_3,E_8,E_9\}$, we get
  {\small
 \[
\left\{\frac{c^5+c^3+c^2+c+1}{c^2},   \frac{(c+1)^4}{c^2+c+1},\frac{(c
   +1)^4}{c},\frac{c   (c+1)^2}{c^2+c+1}\right\},
 \]
 }
 and none of these values can be zero (if $c\neq 0,1$ and $c^5+c^4+c^3+c^2+1=0$).
 
 If $p_0=p_3$, then  $b=\frac{1}{ac}$, and so, $a=\frac{c}{c^2+c+1}, b=\frac{c^2+c+1}{c^2}$, along with $ c^5+c^3+c^2+c+1=0$, which used in $E_1, E_2, E_4, E_5, E_6, E_7, E_8$ gives us
  {\small
 \begin{align*}
 &\left\{\frac{c
   \left(c^5+c^4+c^3+c^2+1\right)}{\left(c^2+c+1\right)^2},\frac{(c+1)^4}{c^2+c+1},\frac{(c+1)^3}{c^2+c+1},\right.\\
   &\qquad\qquad\qquad \left.\frac{c+1}{c^2},\frac{(c+1)^3}{c},\frac{(c+1) \left(c^3+c+1\right)}{c^3},0\right\}.
   \end{align*}
   }
   We see that the only possibility is for $p_0=p_3=p_8$, and we can find $x,y$ values such that $p_0=p_3=p_8=p_{10}$ (under  $c^5+c^3+c^2+c+1=0$, $n\equiv 0\pmod 5$) in the following way.
   The solutions $x,y$ for $C_{10}$ can be found by taking $b=\frac1{ac}$ and arriving at $a^2 c^3 + ac (1 + c +  c^2 ) x + x^2=0$,
 $a c + (1 + c^2) x + c^2 y=0$, for $E_{10}=E_{10}'=0$. If $\Tr\left(\frac{c}{c^2+c+1}\right)=0$, then there is a solution $x$ of the first equation, and consequently for $y$, using the second linear equation. We argue now why the trace is 0. We simply use the equation $c^5+c^4+c^3+c^2+1=0$ and write it as $\frac{c^2}{(c^2+c+1)^2}=\frac{1}{c^3+c}=\frac{1}{c}+\frac{1}{c+1}+\frac{1}{(c+1)^2}=\frac{1}{c}$. Now, using the reciprocal polynomial of $c^5+c^4+c^3+c^2+1$, we find that $\Tr\left(\frac{1}{c}\right)=0$.
 Thus,   ${_c}\cB_G(a,b)=5$ for $n\equiv 0\pmod 5$. 
 
As an example for this argument, we take $n=5$. The solutions for the $c$-boomerang system  and the  parameters ($c,a,b,x,y)$ of the previous argument ($\alpha$ is a primitive root, corresponding to the primitive polynomial $x^5+x^2+1$)  are (we used SageMath to find these solutions, in the finite field $F_{2^5}$ given by the primitive polynomial $x^5+x^2+1$): 
\begin{align*}
& (\alpha^4 + \alpha^2, \alpha^4, \alpha^3 + \alpha^2, 0, \alpha^4 + 1):\text{ solution from the case $C_0$}\\
&(\alpha^4 + \alpha^2, \alpha^4, \alpha^3 + \alpha^2, \alpha^3 + \alpha^2 + \alpha, \alpha^4 + 1); \text{ solution from the case $C_{10}$}\\
&(\alpha^4 + \alpha^2, \alpha^4, \alpha^3 + \alpha^2, \alpha^4 + \alpha + 1, \alpha^3 + \alpha^2 + 1); \text{ solution from the case $C_{10}$}\\
& (\alpha^4 + \alpha^2, \alpha^4, \alpha^3 + \alpha^2, \alpha + 1, \alpha); \text{ solution from the case $C_3$}\\
&(\alpha^4 + \alpha^2, \alpha^4, \alpha^3 + \alpha^2, \alpha^3 + \alpha + 1, \alpha^4 + \alpha^3 + \alpha + 1); \text{ solution from the case $C_8$}.
\end{align*}

   If $p_0=p_4$, then $a=\frac{1}{c^2},b=c$, which when used in $\{E_6,E_8,E_9\}$ finds the values
    {\small
   \[
   \left\{\frac{(c+1)^5}{   c^3},\frac{(c+1)^3}{c^2},c (c+1)^2\right\},
   \]
   }
   and none of these will work.
   If $p_0=p_6$, then $1 + c^2 + c^3 + a (1 + c + c^2)=0$. So, if $c^2+c+1=0$, then $c^3+c^2+1=0$, and so $c=0,1$, which is not allowed.Thus, 
   $a=\frac{c^3+c^2+1}{c^2+c+1},b=\frac{c^2+c+1)}{c(c^3+c^2+1)}$ and $E_0=0$ implies $ c^5+c^4+c^2+c+1=0$ (thus, $n\equiv 0\pmod 5$), which used in $\{ E_8, E_9\}$ gives us 
    {\small
   \[
\left\{(c+1)^3,\frac{c^2 (c+1)^2}{c^3+c^2+1}\right\},
   \]
   }
   and surely, none will work.
   
   If $p_0=p_7$, then $a=\frac{1}{c^3+c^2+1}$ ($c^3+c^2+1\neq 0$, since $E_4=0$ is equivalent to $1 + a (c^3+c^2+1)=0$), $b=\frac{c^3+c^2+1}{c}$ and $c^5+c^4+c^3+c+1=0$. The values used in $\{E_8, E_9\}$ gives us
 {\small
 \[
  \left\{\frac{(c+1   ) \left(c^3+c+1\right)}{c^3+c^2+1},c^3   (c+1)\right\},
 \]
 }
   and none of these values can be zero.
   
   If $p_0=p_8$, then $a^2 c^2 + a (c^2+c+1)+1=0, c + a (c^2+c+1)=0$, $b=\frac{1}{ac}$. If $c^2+c+1=0$, then $a=\frac{1}{c},b=1$. When used in $E_9=0$, we get
   $
c (c+1)=0,
   $
   which is not allowed.
   Thus $c^2+c+1\neq 0$. Then $a=\frac{c}{c^2+c+1},b=\frac{c^2+c+1}{c^2}$ and $c^5+c^3+c^2+c+1=0$ (thus, $n\equiv 0\pmod 4$), which put into $E_9=0$, gives us
   $
(c+1)^2=0.
   $
  There is no need to check for further equalities involving $p_0$.
 
 If $p_1=p_2$, then $b=\frac{c}{a}$ and $E_1=0,E_2=0$ become
 $a^2 + c^2 + a (1 + c + c^2)=0$, $c + a (1 + c + c^2)=0$. Thus, $a=\frac{c}{c^2+c+1},b=c^2+c+1$, and $E_1=0$ is transformed into $c^5+c^4+c^3+c^2+1=0$ (which can always happen for some $c$ if and only if $n\equiv 0\pmod 5$:  that is easy to see since the underlying irreducible polynomial $ x^5+x^4+x^3+x^2+1$ is a divisor of $x^{2^n-1}+1$ if and only if $n\equiv 0\pmod 5$). The values $a=\frac{c}{c^2+c+1},b=c^2+c+1$ put back into $E_3,E_4,E_6,E_8,E_9$ (we use here the fact that $p_0$ and $p_1,p_5,p_7$ cannot be part of a winning ``quintuple'', via the prior discussion) renders
 {\small
\[
\left\{(c+1)^2,(c+1)^3,\frac{c   (c+1)^3}{c^2+c+1},\frac{(c+1)^4}{c^2+c+1},0\right\}.
\]
}
 Thus $p_1=p_2=p_9$ and we can find $(x,y)$, as before, such that $p_1=p_2=p_9=p_{10}$, under $n\equiv 0\pmod 5$.

If $p_1=p_3$, then $a=1,b=c$, which when used in $E_4,E_6,E_8,E_9$ gives $c=1$.
If $p_1=p_4$, then $a=\frac{c^3+c+1}{c(c^2+c+1)}$ (it is easy to see that $c^2+c+1\neq 0$) and $b=\frac{c^2(c^2+c+1)}{c^3+c+1}$ and $c^5+c^4+c^3+c+1=0$.  Used in $E_6,E_9$ renders $c=1$, which is not allowed. If $p_1=p_5$, then $a=\frac{c^3}{c^3+c+1}, b=\frac{c^3+c+1}{c^2}$  and $c^5+c^4+c^2+c+1=0$. Further, using these values in $\{E_6,E_8,E_9\}$, gives
{\small
\[
\left\{\frac{c^2 (c+1)   \left(c^2+c+1\right)}{c^3+c+1},\frac{(c+1)
   \left(c^4+c^3+c^2+c+1\right)}{c^3+c+1},\frac{   (c+1) \left(c^3+c^2+1\right)}{c^2}\right\},
\]
}
but these values cannot be zero under our conditions.
If $p_1=p_6$, then  $a=c^2,b=\frac{1}{c}$, but these values, in $E_7,E_8,E_9$ will not work. If $p_1=p_7$, then $a=c,b=1$, which also imply that $E_8=c+1,E_9=c(c+1)$, and consequently, nonzero. If $p_1=p_8$, then $a=\frac{c^2+c+1}{c}, b=\frac{c^2}{c^2+c+1}$, which is not going to vanish $E_9$.
 
We now assume that $p_2=p_3$. Then, by adding $(a+1)E_2$ and $aE_3$, we get $a + a^2 + a c + a^2 c + c^2 + b (a c + a^2 c + a^3 c)=0$. If $1 + a  + a^2 =0$, then we need $a + a^2 + a c + a^2 c + c^2=0$, that is, $c^2+c+1=0$. Thus, either $a=c$ or $a=\frac{1}{c}$. If $a=c$, and $c^2+c+1=0$ (thus, $n\equiv 0\mod 2$), then $E_2=E_3=0$, gives $1 + c^2 + b c^3 + b^2 c^3 =1 + c + b c + b^2 (c^2 + c^3)=0$.
 If further $p_2=p_3=p_4$, then adding $E_2+E_4=0$ gives $b=1$, which when used in $E_2$ implies $c=1$. 
If $p_2=p_3=p_5$, then, adding $E_2=0,E_5=0$ renders $c=0$, which is impossible.
If $p_2=p_3=p_6$, then adding $E_2,E_6$ gives $b=\frac{1}{c^2}$, which, when used in either $E_2,E_6$ gives $c=1$. 
If $p_2=p_3=p_7$, then $E_3 + c E_7=0$ renders $b=0$, which when used back into the expressions $E_2,E_3,E_7$ renders $c=0,1$.
Surely, $p_2$ cannot equal $p_8$, because of the duality.
There is no need to check further equalities containing $p_2=p_3$, under the assumption $a=c$ and $c^2+c+1=0$. Now, if $c^2+c+1=0$ and $a=\frac{1}{c}$, $p_2=p_3=p_4$ gives $b^2+b+c=0$, which when used in $E_4$ implies $b=\frac{1}{c}$, but then $c=0,1$, from $E_4$. In a similar fashion, $p_2=p_3=p_5$, $p_2=p_3=p_6$, $p_2=p_3=p_7$, $p_2=p_3=p_8$, lead to some contradiction.  There is no need to check further equalities containing $p_2=p_3$, under the assumption $a=\frac{1}{c}$ and $c^2+c+1=0$. We now assume that $a^2+a+1\neq 0$, and so, $b=\frac{a + a^2 + a c + a^2 c + c^2 }{a c (1+a+a^2)}$. If $p_2=p_3=p_4$, then $a=1$, which when used back into $E_1$ gives $c=0,1$. There is nothing tricky about the rest of the cases containing $p_2$, so we move on. Similarly, for $p_3=p_4$, $p_3=p_5$.

We assume now $p_3=p_6$. Adding $a E_3+(a+1)E_6$ gives $b=\frac{a + a^2 + a c + a^2 c + c^2}{a (1 + a^2 + a c)}$. Continuing to add the other possible cases, $p_7$, etc., and based upon observations, we can take $a=c+1,b=\frac{c^2+c+1}{c+1}$, where $1 + c + c^2 + c^3 + c^4=0$ (which can always happen for some $c$ when $n\equiv 0\pmod 4$), which will give us five solutions with $p_3=p_6=p_7=p_{10}$,  and so ${_c}\cB_G(a,b)=5$ for $n\equiv 0\pmod 4$.
 In a similar fashion, taking $a=b=c+1$, and $c$ such that $c^3+c^2+1=0$ (thus, $n\equiv 0\pmod 3$), we see that $p_3=p_7=p_8=p_{10}$, so ${_c}\cB_G(a,b)=5$ for $n\equiv 0\pmod 3$.
 With the same approach, with $a=\frac{c+1}{c}, b=\frac{c^2+c+1}{c(c+1)}$, and  $c^4+c^3+c^2+c+1=0$ (so, $n\equiv 0\pmod 4$), then $p_4=p_5=p_9=p_{10}$, so ${_c}\cB_G(a,b)=5$ in this case. 
 The last claims of Theorem~\ref{thmBCT} on the bound being attained are  shown.

 \subsection{Some computational data}
 
 Below, we will compute the maximum $c$-boomerang uniformity of the $(0,1)$-swapped binary inverse function in the fields $\F_{2^2}, \F_{2^3}, \F_{2^4}, \F_{2^5}$, determined by the primitive polynomials $X^2+X+1$, $X^3+X+1$, $X^4+X+1$, respectively, $X^5+X^2+1$, each of some primitive root $\alpha$.
 
 The values of $(c,a,b,x,y)$ achieving the  (maximum)  $c$-boomerang uniformity of $1$ for $n=2$ are:
   \allowdisplaybreaks
 {\footnotesize
 \begin{align*}
 & (\alpha, \alpha, 0,\alpha, 0), (\alpha, \alpha, \alpha, \alpha, \alpha), (\alpha, \alpha, \alpha + 1, \alpha, \alpha + 1), (\alpha, \alpha, 1, \alpha, 1), (\alpha, \alpha + 1, 0, 0, \alpha + 1),\\ &(\alpha, \alpha + 1, \alpha, 0, 1), (\alpha, \alpha + 1, \alpha + 1, 0, 0), (\alpha, \alpha + 1, 1, 0, \alpha), (\alpha, 1, 0, \alpha + 1, \alpha), \\
 & (\alpha, 1, \alpha, \alpha + 1, 0), (\alpha, 1, \alpha + 1, \alpha + 1, 1), (\alpha, 1, 1, \alpha + 1, \alpha + 1), (\alpha + 1, \alpha, 0, 0, \alpha), \\
 & (\alpha + 1, \alpha, \alpha, 0, 0), (\alpha + 1, \alpha, \alpha + 1, 0, 1),  (\alpha + 1, \alpha, 1, 0, \alpha + 1), (\alpha + 1, \alpha + 1, 0, \alpha + 1, 0), \\
 & (\alpha + 1, \alpha + 1, \alpha, \alpha + 1, \alpha),  (\alpha + 1, \alpha + 1, \alpha + 1, \alpha + 1, \alpha + 1), (\alpha + 1, \alpha + 1, 1, \alpha + 1, 1), \\
 & (\alpha + 1, 1, 0, \alpha, \alpha + 1), (\alpha + 1, 1, \alpha, \alpha, 1), (\alpha + 1, 1, \alpha + 1, \alpha, 0), (\alpha + 1, 1, 1, \alpha, \alpha).
 \end{align*}
 }
 
  The values of $(c,a,b,x,y)$ achieving the (maximum)  $c$-boomerang uniformity of $4$ for $n=3$ are:
    \allowdisplaybreaks
  {\footnotesize
  \begin{align*}
   &  \text{for } (c,a,b)= (\alpha, \alpha^2 + 1, 0),\text{ the solutions $(x,y)$ are}:\\
 & ( 0, \alpha^2 + 1), (\alpha, 0), (\alpha^2 + \alpha, \alpha + 1), (\alpha^2 + 1, \alpha^2 + \alpha + 1), \\
 &  \text{for } (c,a,b)= (\alpha^2, \alpha^2 + \alpha + 1, 0),\text{ the solutions $(x,y)$ are}:\\
 & (0, \alpha^2 + \alpha + 1), (\alpha, \alpha^2 + 1), ( \alpha^2, 0), (\alpha^2 + \alpha + 1, \alpha + 1);\\
 &  \text{for } (c,a,b)= (\alpha + 1, \alpha + 1, 0),\text{ the solutions $(x,y)$ are}:\\
 &  (0, \alpha^2 + \alpha), ( \alpha + 1, 0), (\alpha^2 + \alpha + 1, \alpha^2), (\alpha^2 + 1, \alpha + 1)\\
 &  \text{for } (c,a,b)= (\alpha^2 + \alpha, \alpha + 1, 0),\text{ the solutions $(x,y)$ are}:\\
 &  ( 0, \alpha + 1), (\alpha^2, \alpha^2 + \alpha + 1), (\alpha + 1, \alpha^2 + 1), (\alpha^2 + \alpha, 0); \\
 &  \text{for } (c,a,b)= (\alpha^2 + \alpha + 1, \alpha^2 + \alpha + 1, 0),\text{ the solutions $(x,y)$ are}:\\
 & ( 0, \alpha^2),  (\alpha + 1, \alpha^2 + \alpha + 1), (\alpha^2 + \alpha + 1, 0),  (\alpha^2 + 1, \alpha);\\
 &  \text{for } (c,a,b)= (\alpha^2 + 1, \alpha^2 + 1, 0),\text{ the solutions $(x,y)$ are}:\\
& ( 0, \alpha), (\alpha + 1, \alpha^2 + \alpha),  (\alpha^2 + \alpha + 1, \alpha^2 + 1), (\alpha^2 + 1, 0).
  \end{align*}
  }

  The values of $(c,a,b,x,y)$ achieving the (maximum) $c$-boomerang uniformity of $5$ for $n=4$ are:
  \allowdisplaybreaks
 {\footnotesize
  \begin{align*}
& \text{for }(c,a,b)=(\alpha^3, \alpha^3 + \alpha^2 + \alpha, \alpha^3 + \alpha^2), \text{ the solutions $(x,y)$ are}: \\
& (0, \alpha^3 + \alpha^2 + 1), (\alpha, \alpha^3 + \alpha^2 + \alpha + 1), (\alpha^2, \alpha + 1), (\alpha^3 + \alpha^2 + 1, \alpha^3 + \alpha^2),  (1, \alpha^3 + \alpha); \\
&\text{for } (c,a,b)=(\alpha^3, \alpha^3 + 1, \alpha^3 + \alpha),\text{ the solutions $(x,y)$ are}:\\
&  (\alpha, \alpha + 1), (\alpha^3, \alpha^2 + 1),  (\alpha^3 + \alpha^2, \alpha^3 + \alpha + 1), (\alpha^3 + \alpha, 1), (\alpha^3 + 1, \alpha);\\
& \text{for } (c,a,b)= (\alpha^3 + \alpha^2, \alpha^3 + \alpha + 1, \alpha^3 + \alpha^2 + \alpha + 1),\text{ the solutions $(x,y)$ are}:\\
& (0, \alpha^3 + \alpha^2 + \alpha),  (\alpha^2, \alpha^3 + \alpha), (\alpha + 1, \alpha^2 + 1),  (\alpha^3 + \alpha^2 + \alpha, \alpha^3 + \alpha^2 + \alpha + 1), (1, \alpha^3);\\
&  \text{for } (c,a,b)= (\alpha^3 + \alpha^2, \alpha^3 + \alpha^2 + 1, \alpha^3),\text{ the solutions $(x,y)$ are}:\\
 &(\alpha^2, \alpha^2 + 1), (\alpha^3, 1), (\alpha^3 + \alpha^2, \alpha),  (\alpha^3 + \alpha^2 + \alpha + 1, \alpha^3 + 1), (\alpha^3 + \alpha^2 + 1, \alpha^2);\\
 &  \text{for } (c,a,b)= (\alpha^3 + \alpha, \alpha^3 + \alpha + 1, \alpha^3 + \alpha^2 + \alpha + 1),\text{ the solutions $(x,y)$ are}:\\
&   ( \alpha^3, \alpha^3 + \alpha^2 + \alpha),   (\alpha^3 + \alpha + 1, \alpha^2 + 1),  (\alpha^2 + 1, \alpha^2),  (\alpha^3 + \alpha, \alpha + 1),
  (\alpha^3 + \alpha^2 + \alpha + 1, 1); \\
&  \text{for } (c,a,b)= (\alpha^3 + \alpha, \alpha^3 + \alpha^2 + 1, \alpha^3),\text{ the solutions $(x,y)$ are}:\\
& (0, \alpha^3 + 1), (\alpha, \alpha^2), (\alpha^2 + 1, \alpha^3 + \alpha^2),(\alpha^3 + 1, \alpha^3), (1, \alpha^3 + \alpha^2 + \alpha + 1);\\
 &  \text{for } (c,a,b)= (\alpha^3 + \alpha^2 + \alpha + 1, \alpha^3 + \alpha^2 + \alpha, \alpha^3 + \alpha^2), \text{ the solutions $(x,y)$ are}:\\
&  (\alpha + 1, \alpha), (\alpha^3 + \alpha^2, 1),(\alpha^3 + \alpha, \alpha^3 + \alpha^2 + 1),(\alpha^3 + \alpha^2 + \alpha, \alpha + 1), 
(\alpha^3 + \alpha^2 + \alpha + 1, \alpha^2); \\
&  \text{for } (c,a,b)= (\alpha^3 + \alpha^2 + \alpha + 1, \alpha^3 + 1, \alpha^3 + \alpha),\text{ the solutions $(x,y)$ are}:\\
& (0, \alpha^3 + \alpha + 1), (\alpha^3 + \alpha, \alpha + 1, \alpha^3),  (\alpha^3 + \alpha + 1, \alpha^3 + \alpha), ( \alpha^2 + 1, \alpha),(1, \alpha^3 + \alpha^2).
  \end{align*}
  }

   The values of $(c,a,b,x,y)$ achieving the (maximum) $c$-boomerang uniformity of $5$ for $n=5$ are:
  \allowdisplaybreaks
 {\scriptsize
  \begin{align*}
& \text{for }(c,a,b)=(\alpha^3, \alpha^4 + \alpha^3 + \alpha + 1, \alpha + 1), \text{ the solutions $(x,y)$ are}: \\
& (\alpha^2, 0), (\alpha^3, \alpha^4 + \alpha), (\alpha^3 + \alpha, \alpha^4 + \alpha^3 + \alpha), ( \alpha^4 + 1, \alpha^4 + \alpha + 1), (\alpha^4 + \alpha^3 + \alpha + 1, 1); \\
 & \text{for }(c,a,b)=(\alpha^3 + \alpha, \alpha, \alpha^2 + 1), \text{ the solutions $(x,y)$ are}: \\
 & ( \alpha, 1), (\alpha^4, 0), (\alpha^3 + \alpha, \alpha^3 + 1), (\alpha^3 + \alpha^2 + \alpha, \alpha + 1), (\alpha^3 + \alpha^2, \alpha^3);\\
 & \text{for }(c,a,b)=(\alpha^4 + \alpha^2, \alpha^4, \alpha^3 + \alpha^2), \text{ the solutions $(x,y)$ are}: \\
 & ( 0, \alpha^4 + 1), (\alpha^3 + \alpha^2 + \alpha, \alpha^4 + \alpha^3 + \alpha^2 + \alpha + 1), (\alpha^4 + \alpha + 1, \alpha^4 + \alpha^3 + \alpha^2 + \alpha), (\alpha + 1, 1), (\alpha^3 + \alpha + 1, \alpha^4);\\
 & \text{for }(c,a,b)=(\alpha^3 + \alpha^2 + \alpha, \alpha^2, \alpha^4 + 1), \text{ the solutions $(x,y)$ are}: \\
 & (\alpha^2, 1), ( \alpha^3 + \alpha^2 + 1, 0), (\alpha^4 + \alpha^3 + \alpha, \alpha^3 + \alpha), (\alpha^3 + \alpha^2 + \alpha, \alpha^3 + \alpha + 1), (\alpha^4 + \alpha^3 + \alpha^2 + \alpha, \alpha^2 + 1);\\
 & \text{for }(c,a,b)=(\alpha^4 + \alpha^3 + \alpha^2 + 1, \alpha^3 + \alpha^2 + 1, \alpha^4 + \alpha^3 + \alpha), \text{ the solutions $(x,y)$ are}: \\
 & (0, \alpha^3 + \alpha^2), (\alpha^3, \alpha^4 + \alpha + 1), (\alpha^2 + 1, 1), (\alpha^3 + \alpha^2 + \alpha + 1, \alpha^3 + \alpha^2 + 1), (\alpha^4 + \alpha^3 + \alpha^2 + \alpha, \alpha^4 + \alpha);\\
  & \text{for }(c,a,b)=(\alpha^4 + \alpha + 1, \alpha^3 + \alpha^2 + 1, \alpha^4 + \alpha^3 + \alpha), \text{ the solutions $(x,y)$ are}: \\
 & (\alpha, 0), (\alpha^3, \alpha^3 + \alpha^2), (\alpha^2 + 1, \alpha^4 + \alpha^3 + \alpha^2 + \alpha), (\alpha^3 + \alpha^2 + 1, 1), (\alpha^4 + \alpha + 1, \alpha^4 + \alpha^3 + \alpha^2 + \alpha + 1);\\
   & \text{for }(c,a,b)=(\alpha^2 + \alpha, \alpha^2, \alpha^4 + 1), \text{ the solutions $(x,y)$ are}: \\
 &  (0, \alpha^2 + 1), (\alpha^3 + \alpha, \alpha^3 + \alpha^2 + \alpha + 1), (\alpha^4 + \alpha^3 + \alpha, 1), (\alpha^4 + \alpha^3 + \alpha^2 + \alpha, \alpha^3 + \alpha^2 + \alpha), (\alpha^3 + 1, \alpha^2);\\
  & \text{for }(c,a,b)=(\alpha^4 + \alpha^3 + \alpha^2 + \alpha, \alpha^4, \alpha^3 + \alpha^2), \text{ the solutions $(x,y)$ are}: \\
 & (\alpha^4, 1), (\alpha^4 + \alpha^3 + \alpha + 1, 0), (\alpha^4 + \alpha + 1, \alpha^4 + 1), (\alpha + 1, \alpha^3 + \alpha^2 + \alpha), (\alpha^4 + \alpha^3 + \alpha^2 + \alpha, \alpha^3 + \alpha^2 + \alpha + 1); \\
   & \text{for }(c,a,b)=(\alpha^4 + \alpha^3 + 1, \alpha, \alpha^2 + 1), \text{ the solutions $(x,y)$ are}: \\
 &(0, \alpha + 1), (\alpha^3, \alpha^3 + \alpha + 1), (\alpha^3 + \alpha^2 + \alpha, \alpha^3 + \alpha), (\alpha^3 + \alpha^2, 1), (\alpha^4 + \alpha, \alpha); \\
   & \text{for }(c,a,b)=(\alpha^4 + \alpha^2 + \alpha, \alpha^4 + \alpha^3 + \alpha + 1, \alpha + 1), \text{ the solutions $(x,y)$ are}: \\ 
 &(0, \alpha^4 + \alpha^3 + \alpha), (\alpha^3 + \alpha, \alpha^3), (\alpha^4 + 1, 1), (\alpha^4 + \alpha^3 + \alpha^2 + \alpha + 1, \alpha^4 + \alpha^3 + \alpha + 1), ( \alpha^4 + \alpha + 1, \alpha^3 + 1).
  \end{align*}
  }
 
 \end{document}